\documentclass{amsart}
\usepackage{amsmath,amssymb,amscd,latexsym}

\usepackage[all]{xy}

\usepackage{graphicx}

\usepackage{mathrsfs}

\usepackage{stmaryrd}

\usepackage{enumitem}
\usepackage{nicefrac}
\usepackage{multirow}

\usepackage{tikz}
\usepackage{tikz-cd}

\usepackage{hyperref}

\usepackage[normalem]{ulem}

\definecolor{NTNUblue}{RGB}{0,80,158}
\definecolor{NTNUbluesupport}{RGB}{62,98,138}
\definecolor{NTNUorange}{RGB}{239,129,20}

\newcommand{\R}{{\mathbb{R}}}
\newcommand{\Z}{{\mathbb{Z}}}
\newcommand{\Gr}[2]{\widetilde{\mathrm{Gr}_{#1}}(\mathbb{R}^{#2})}

\newcommand{\Spin}{\mathrm{Spin}}
\newcommand{\Sq}{\mathrm{Sq}}

\renewcommand{\O}{\mathbb{O}}
\newcommand{\ImO}{\mathrm{Im}\;\mathbb{O}}

\newcommand{\rz}{\mathbb{R}/\mathbb{Z}}

\newcommand{\trz}{\tau_{\rz}}

\newtheorem{theorem}{Theorem}[section]
\newtheorem{lemma}[theorem]{Lemma}
\newtheorem{prop}[theorem]{Proposition}
\newtheorem{cor}[theorem]{Corollary}

\newtheorem*{theorem*}{Theorem}

\theoremstyle{definition}
\newtheorem{defn}[theorem]{Definition}

\newtheorem{remark}[theorem]{Remark}

\usepackage{adjustbox}

\begin{document}

\author{Eiolf Kaspersen}
\address{Department of Mathematical Sciences, NTNU, NO-7491 Trondheim, Norway}
\email{eiolf.kaspersen@ntnu.no}

\author{Gereon Quick}
\thanks{The second-named author was partially supported by the RCN Project No.\,313472.}
\address{Department of Mathematical Sciences, NTNU, NO-7491 Trondheim, Norway}
\email{gereon.quick@ntnu.no}

\title[Geometric cohomology classes for SPIN$(7)$ and SPIN$(8)$]{Geometric representation of cohomology classes for the Lie groups SPIN$(7)$ and SPIN$(8)$}

\date{}

\begin{abstract}
By constructing concrete complex-oriented maps we show that the eight-fold of the generator of the third integral cohomology of the spin groups $\Spin(7)$ and $\Spin(8)$ is in the image of the Thom morphism from complex cobordism to singular cohomology, while the generator itself is not in the image. 
We thereby give a geometric construction for a nontrivial class in the kernel of the differential Thom morphism of Hopkins and Singer for the Lie groups $\Spin(7)$ and $\Spin(8)$. 
The construction exploits the special symmetries of the octonions. 
\end{abstract}
\subjclass{\rm 57R77, 57T10, 55N22, 55S05, 55S10} 
\keywords{Thom morphism, complex cobordism, spin groups, differential cohomology, octonions}

\maketitle

\section{Introduction}

Let $\Spin(n)$ be the double cover of the real special orthogonal group $SO(n)$. 
The group $\Spin(n)$ is a compact smooth manifold of dimension $\frac{1}{2}n(n-1)$ and plays an important role in many areas of mathematics and mathematical physics. 
It is a natural question which elements of its cohomology can be represented by smooth manifolds. 
More precisely, based on Quillen's work in \cite{quillen}, we may ask which of the generators of $H^*(\Spin(n);\Z)$ are in the image of the Thom morphism 
\[
\tau \colon MU^*(\Spin(n)) \longrightarrow H^*(\Spin(n);\Z),
\]
where $MU$ denotes complex cobordism. 
While it is well-known that the Thom morphism is, in general, not surjective in degrees $\ge 3$  
(see for example \cite[Theorem 2.2]{totaro}), 
$\tau$ is indeed surjective for the spin groups $\Spin(n)$ for $n\le 6$ by \cite[Proposition 3.4]{KQ}. 
For $n\ge 7$, however, the Steenrod square $\Sq^3$ does not vanish on the mod $2$-reduction of the non-torsion generator $\gamma_3 \in H^3(\Spin(n);\Z)$.  
By \cite{cartanSem}, this implies that $\gamma_3$ cannot be represented by a smooth manifold (see \cite{KQ} for a further discussion and see for example \cite{nishimoto} and \cite{yagita2, yagita3} for related computations for Lie groups). 
In this note, we show that $8\cdot \gamma_3$ is, in fact, in the image of $\tau$, and 
we construct a concrete geometric representative. 
More precisely, we prove the following result (see Theorem \ref{thm:spin} below): 

\begin{theorem}\label{thm:main_theorem_intro} 
Let $\Gr{k}{n}$ denote the Grassmannian of oriented $k$-planes in $\R^n$ na dlet $S^1$ denote the unit circle.  
Let $\gamma_3$ be a non-torsion generator of $H^3(\Spin(7);\Z)$. 
There is a complex-oriented map 
\begin{align*}
\varphi = f_7 \times f_5 \colon \Gr{2}{7} \times S^1 \times \Gr{2}{5} \times S^1 \longrightarrow \Spin(7) 
\end{align*}
such that the Thom morphism $MU^3(\Spin(7)) \rightarrow H^3(\Spin(7);\Z)$ maps the element $[\varphi]$ represented by $\varphi$ to $8 \cdot \gamma_3 $.  
\end{theorem}

\begin{remark}
In \cite[Theorem 4.3]{KQ}, we constructed a geometric representative of two times the generator of $H^3(SO(5);\Z)$. 
This construction then generalizes to provide geometric representatives of higher multiples of the generators of $H^{4k+3}(SO(n);\Z)$ for $k \ge 0$ and $n \ge k+1$ (see \cite[Theorem 4.7]{KQ} for the precise formula). 
For spin groups, however, we do not know how to extend the construction of a corresponding map for $\Spin(n)$ with $n\ge 9$, 
despite the fact that such maps exist for all $n\ge 7$ by \cite[Proposition 3.4]{KQ}. 
The key features that make the construction of the map in Theorem \ref{thm:main_theorem_intro} possible for $\Spin(7)$, and then also for $\Spin(8)$ in Corollary \ref{cor:spin8}, 
are the special symmetries of the octonions which the groups $\Spin(7)$ and $\Spin(8)$ are related to. 
\end{remark}

There are at least two reasons why Theorem \ref{thm:main_theorem_intro} is important.  
On the one hand, we obtain a more complete picture of the geometric representation of the generators of $H^*(\Spin(7);\Z)$ and a geometric interpretation of the failure of the Thom morphism to be surjective in degree $3$.  
On the other hand, Theorem \ref{thm:main_theorem_intro} yields a geometric construction of a nontrivial element of degree $4$ in the kernel of the differential refinement of the Thom morphism of Hopkins--Singer as we will explain next (for a more general and more detailed discussion we refer to \cite[Section 2.3]{KQ}). 
For the smooth manifold $\Spin(7)$, let $\check{MU}(4)^4(\Spin(7))$ and $\check{H}(4)^4(\Spin(7))$ denote the differential refinements of complex cobordism and singular cohomology of Hopkins--Singer in \cite{hs}, respectively. 
By \cite[diagram (4.57)]{hs}, the Thom morphism $\tau \colon MU \to H\Z$ induces a commutative diagram 
\begin{align}\label{eq:map_of_ses_check_MUtoHZ}
\xymatrix{
MU^{3}(\Spin(7))\otimes_{\Z} \rz \ar[d]_-{\trz} \ar[r] & \check{MU}(4)^4(\Spin(7)) \ar[d]^-{\check{\tau}} \\
H^{3}(\Spin(7);\Z)\otimes_{\Z} \rz \ar[r] & \check{H}(4)^4(\Spin(7)) 
}    
\end{align}
in which the horizontal maps are injective. 
This implies that every element in the kernel of $\trz$ induces an element in the kernel of $\check{\tau}$. 
We say that an element in the kernel of $\trz$ or $\check{\tau}$ is \emph{nontrivial} if it is not contained in the respective ideal generated by $MU^{*<0}$. 
We claim that the element $[\varphi] \otimes 1/8 \in MU^3(\Spin(7)) \otimes\rz$ is a nontrivial element in the kernel of $\trz$:  
By Theorem \ref{thm:main_theorem_intro}, we have  
\[
\trz([\varphi] \otimes 1/8)
= \tau([\varphi]) \otimes 1 
= 8 \cdot \gamma_3 = 0.
\]
Hence $[\varphi] \otimes 1/8$ lies in the kernel of $\trz$. 
However, $[\varphi] \otimes 1/8$ is not $0$ in $MU^3(\Spin(7)) \otimes \rz$, since if $[\varphi]$ had been of the form $8 \cdot [\varphi']$ with $[\varphi'] \in MU^3(\Spin(7))$, 
then $[\varphi']$ would map to $\gamma_3$ or $\gamma_3+y$ for some class $y \in H^3(\Spin(7);\Z)$ with $8y=0$. 
The latter is not the case by \cite[Proposition 3.4]{KQ}. 
Moreover, $[\varphi] \otimes 1/8$ does not lie in the ideal generated by $MU^{*<0}$, since otherwise $[\varphi]$ would be in the kernel of $\tau$. 
Since $[\varphi]$ maps to $8 \cdot \gamma_3 \ne 0$ in $H^3(\Spin(7);\Z)$, the latter is not the case. 
Hence Theorem \ref{thm:main_theorem_intro} provides a geometric description of a nontrivial element in the kernel of the differential Thom morphism for $\Spin(7)$: 

\begin{theorem}\label{thm:diff_tau_spin_intro}
The class $\frac{1}{8}[\varphi]$ 
is a nontrivial element in the kernel of 
\begin{align*}
\check{\tau}  \colon \check{MU}(4)^4(\Spin(7)) \longrightarrow \check{H}(4)^4(\Spin(7))  
\end{align*}
which is not contained in the ideal generated by $MU^{*<0}$.  
\end{theorem} 

Elements in the kernel of the differential refinement of the Thom morphism are particularly interesting as they show in which way differential cobordism is a finer invariant than ordinary differential cohomology. 
Finally, we note that in \cite[\S 2.7]{hs} the group $\check{H}(4)^4(M)$, for a certain spin manifold $M$, explains the behavior of an important partition function in mathematical physics. 
We refer to \cite[Example 48]{gradysatiAHSS} for other interesting phenomena in mathematical physics related to the study of the morphisms between generalised differential cohomology theories. 
We do not know of a potential similar application of Theorem \ref{thm:diff_tau_spin_intro} yet. 
We expect, however, that the techniques used to prove Theorem \ref{thm:diff_tau_spin_intro} at least will help to shed new light on such phenomena and on the Abel--Jacobi invariant for complex cobordism of \cite{hausquick}. \\

The paper is organised as follows. 
In section \ref{sec:octonions} we recall some basic facts about the geometry of the octonions and special orthogonal and spin groups. 
In section \ref{sec:map} we construct the maps that feature in Theorem \ref{thm:main_theorem_intro} by using the special symmetries of the octonions. 
The nice properties of oriented Grassmannians and compact Lie groups then imply that the map is complex-oriented. 
In section \ref{sec:homology} we prove our main technical results Theorem \ref{thm:spin} and Corollary \ref{cor:spin8} by computing the effect of the relevant maps in homology via local degrees.


\section{Octonions and \texorpdfstring{$\Spin(7)$}{Spin(7)}}\label{sec:octonions}

We let $\O$ denote the \textit{octonions}, i.e., the normed division algebra with $\R$-basis $\{e_0,e_1,e_2,e_3,e_4,e_5,e_6,e_7\}$ which satisfies the following relations: 
We have $e_0=1$ and $e_i^2 = -1$ for $i \neq 0$, and 
the multiplication of distinct basis elements is given by the following Fano plane as described in \cite{mn}. 

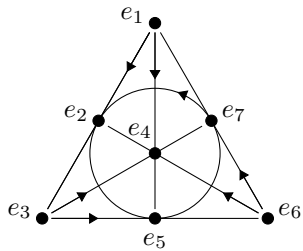
\begin{center}
\begin{tikzpicture}[scale=1.5]
    \node(1)at(0,1.732){};
    \node(2)at(-0.5,0.866){};
    \node(3)at(-1,0){};
    \node(4)at(0,0.577){};
    \node(5)at(0,0){};
    \node(6)at(1,0){};
    \node(7)at(0.5,0.866){};

    \draw[fill=black](1)circle(0.05);
    \draw[fill=black](2)circle(0.05);
    \draw[fill=black](3)circle(0.05);
    \draw[fill=black](4)circle(0.05);
    \draw[fill=black](5)circle(0.05);
    \draw[fill=black](6)circle(0.05);
    \draw[fill=black](7)circle(0.05);

    \draw (1) -- (3);
    \draw (6) -- (3);
    \draw (1) -- (6);
    \draw (1) -- (5);
    \draw (1) -- (3);
    \draw (7) -- (3);
    \draw (2) -- (6);
    \draw (4)circle(0.577);

    \node[xshift=-0.3cm, yshift=0.1cm] at (1) {$e_1$};
    \node[xshift=-0.3cm, yshift=0.1cm] at (2) {$e_2$};
    \node[xshift=-0.3cm, yshift=0.1cm] at (3) {$e_3$};
    \node[xshift=-0.2cm, yshift=0.3cm] at (4) {$e_4$};
    \node[xshift=-0.0cm, yshift=-0.3cm] at (5) {$e_5$};
    \node[xshift=0.3cm, yshift=0.1cm] at (6) {$e_6$};
    \node[xshift=0.3cm, yshift=0.1cm] at (7) {$e_7$};

    \draw[-Triangle] (3) -- (-0.5,0);
    \draw[-Triangle] (3) -- (-0.6,0.23);
    \draw[-Triangle] (1) -- (0,1.232);
    \draw[-Triangle] (1) -- (-0.25,1.299);
    \draw[-Triangle] (6) -- (0.6,0.23);
    \draw[-Triangle] (6) -- (0.75,0.433);
    \draw[-Triangle] (0.201,1.1196) -- (0.2, 1.12);
    
\end{tikzpicture}
\end{center}
For example, we have $e_3 e_2 = -e_1$, where the negative sign corresponds to the fact that the path from $e_3$ to $e_2$ goes against the direction of the arrow. 
We note that several versions of this Fano plane exist, producing different bases for $\O$. 
%
The subspace of $\O$ spanned by $\{e_1,\ldots,e_7\}$ is referred to as the \textit{purely imaginary octonions}, denoted $\ImO$.

By associating the octonions with $\R^8$ and the purely imaginary octonions with $\R^7$, the group $SO(8)$ acts on the octonions $\O$, while $SO(7)$ acts on $\ImO$. 
We then have the following construction of the group $\Spin(7)$, from \cite{imn} and \cite{yokota67}. 
For every $g\in SO(7)$ there exists a unique, up to sign, $\Tilde{g}\in SO(8)$ such that $g(x)\Tilde{g}(y) = \Tilde{g}(xy)$ for all $x,y \in \O$. 
Then we can define $\Spin(7)$ as a subgroup of $SO(8)$ by
\begin{equation*}
    \Spin(7) = \left\{ \Tilde{g}\in SO(8)\; \vert\; g\in SO(7) \right\}.
\end{equation*}
Thus, elements of $\Spin(7)$ can be considered as transformations of $8$-dimensional Euclidean space. The canonical double covering $\Spin(7) \rightarrow SO(7)$ is the map $\Tilde{g} \mapsto g$. For future reference we now state without proof some well-known properties of the octonions \cite[Chapters 6.5 and 6.7]{cs}.

\begin{lemma}\label{properties}
Let $x,y,z\in \O$. Then we have
\begin{itemize}
    \item[\rm{(i)}]  $x(xy)=(xx)y$, \\
                $(yx)x=y(xx)$,
    \item[\rm{(ii)}] $(x(yz))x = x((yz)x) = (xy)(zx)$, \\
                $(x(yx))z = x(y(xz))$,\\
                $y(x(zx)) = ((yx)z)x$.
\end{itemize}
Now let $x,y,z\in \ImO$. Then we have
\begin{itemize}
    \item[\rm{(iii)}] $xy=z \Leftrightarrow x=yz$,
    \item[\rm{(iv)}] $xy = -yx$ if $x \perp y$.
\end{itemize}    
\end{lemma}
We then state and prove one more property of the octonions.
\begin{lemma}\label{associativity}
    Let $x,y,z\in \ImO$ such that $x,y,z,xy$ are mutually orthogonal. Then $x(yz) = -(xy)z$.
\end{lemma}
\begin{proof}
Using Lemma \ref{properties}, we compute
    \begin{align*}
        x(yz) &= -x(zy) = (zy)x = -(xx)((zy)x) = -x(x((zy)x)) \\
        &= -x((xz)(yx)) = x((yx)(xz)) = (x((yx)x))z = -(xy)z.
    \end{align*}\\[-28pt]
\end{proof}

We will use the following notation. 

\begin{defn}
We call a pair $(L,\sigma)$ consisting of a $2$-dimensional sub-vector space $L \subseteq \R^n$ with an orientation $\sigma$ of $L$ an \emph{oriented plane}. 
When $x$ and $y$ are orthonormal vectors in $\R^n$, we write $[x,y]$ for the oriented plane spanned by $x$ and $y$ where the orientation is given by the ordering of the two vectors, i.e., 
$[x,y]$ and $[y,x]$ have opposite orientations. 
\end{defn}

\begin{defn}
Let $(L,\sigma)\subseteq \R^n$ be an oriented plane, and let $t\in \R$. We write $r_{L,\sigma,t}$ for the rotation of the plane $L$ by the angle $t$ along the orientation $\sigma$.
\end{defn}

We note that if $(L,\sigma)=[x,y]$, then $r_{L,\sigma,t}$ is given by
\begin{align*}
    x &\longmapsto \cos{t}x + \sin{t}y \\
    y &\longmapsto -\sin{t}x + \cos{t}y. 
\end{align*}
We now extend $r_{l,\sigma,t}$ to a transformation of $\R^8$.

\begin{defn}
Let $(L,\sigma)\subset \R^8$ be an oriented plane, $t\in \R$ and $z \in \R^8$. 
For $z \in \R^8$, let $z_1\in L$ and $z_2 \in L^{\perp}$ be the unique vectors such that $z=z_1 + z_2$. 
Then we define $\psi_{(L,\sigma),t}\in SO(8)$ by
    \begin{equation*}
        \psi_{(L,\sigma),t}(z) = r_{L,\sigma,t}(z_1) + z_2.
    \end{equation*}
    We will refer to an element of the form $\psi_{(L,\sigma),t}$ as a \textit{rotation}.
\end{defn}


\section{Constructing the map}\label{sec:map}

Let $\Gr{k}{n}$ denote the Grassmannian of oriented $k$-planes in $\R^n$. 
We think of the circle $S^1$ as embedded in the complex plane, and we will write its points on the form $e^{it}$ parameterized by $t\in \R$. 
We now construct a map which we subsequently prove is well-defined.

\begin{defn}\label{spinmap}
For $[x,y]\in\Gr{2}{7}$, let $w$ denote a unit vector in $\R^8$ such that $w \perp \mathrm{Span}\{e_0,x,y,xy\}$. 
We then define
    \begin{align*}
    f_7\colon \Gr{2}{7} \times S^1 &\longrightarrow SO(8) \\
    ([x,y],e^{it}) &\longmapsto \psi_{[x,y],t}\cdot\psi_{[e_0,xy],t}\cdot\psi_{[w,w(xy)],t}\cdot\psi_{[wx,wy],t}.
\end{align*}
\end{defn}
\begin{lemma}\label{lem:f7_well_defined}
The map $f_7$ is well-defined.
\end{lemma}
\begin{proof}
Firstly, we show that the map is invariant under the choice of basis $[x,y]$. A different basis for the same oriented plane is given by
\begin{equation*}
        [x',y'] = [\cos{s}x + \sin{s}y,-\sin{s}x + \cos{s}y], \quad s\in \R.
\end{equation*}
We then see that
\begin{align*}
        x'y' &= (\cos{s}x + \sin{s}y)(-\sin{s}x + \cos{s}y) \\&= \cos^2(s)xy - \sin^2(s)yx - \cos{s}\sin{s}x^2 + \cos{s}\sin{s}y^2 = xy.
\end{align*}
Thus, we have
\begin{align*}
    [x',y'] & = [x,y], ~
    [e_0,x'y'] = [e_0,xy], ~  
    [w,w(x'y')] = [w,w(xy)], \\
    [wx',wy'] &= [\cos{s}wx + \sin{s}wy, -\sin{s}wx + \cos{s}wy] = [wx,wy], 
\end{align*}
which implies that
\begin{equation*}
\psi_{[x,y],t}\cdot\psi_{[e_0,xy],t}\cdot\psi_{[w,w(xy)],t}\cdot\psi_{[wx,wy],t} = \psi_{[x',y'],t}\cdot\psi_{[e_0,x'y'],t}\cdot\psi_{[w,w(x'y')],t}\cdot\psi_{[wx',wy'],t}.
\end{equation*}
    
Next, we show that $f_7$ is invariant under the choice of $w$. Since the rotations $\psi_{[x,y],t}$ and $\psi_{[e_0,xy],t}$ do not depend on the choice of $w$, it suffices to show that the map
\begin{equation*}
        \widetilde{f_7}\colon([x,y],e^{it})\mapsto \psi_{[w,w(xy)],t}\cdot \psi_{[wx,wy],t}
\end{equation*}
is well-defined. Given some $w$, we claim that
\begin{equation*}
    \mathcal{B} = \{e_0,x,y,xy,w,wx,wy,w(xy)\}
\end{equation*}
is an orthogonal basis for $\O$. The vectors $e_0,x,y,xy,w$ are mutually orthogonal by construction. It follows from Lemmas \ref{properties} and \ref{associativity} that the remaining ones are also mutually orthogonal. For example, we have
\begin{equation*}
    -wx \perp -e_0 \implies -(wx)y \perp -y \implies w(xy) \perp y.
\end{equation*}
Thus, $\mathcal{B}$ is an orthogonal basis. Furthermore, Lemmas \ref{properties} and \ref{associativity} imply that the multiplication of basis elements is given by
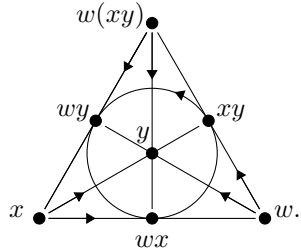
\begin{equation}\label{Bfano}
\begin{tikzpicture}[scale=1.5]
    \node(1)at(0,1.732){};
    \node(2)at(-0.5,0.866){};
    \node(3)at(-1,0){};
    \node(4)at(0,0.577){};
    \node(5)at(0,0){};
    \node(6)at(1,0){};
    \node(7)at(0.5,0.866){};

    \draw[fill=black](1)circle(0.05);
    \draw[fill=black](2)circle(0.05);
    \draw[fill=black](3)circle(0.05);
    \draw[fill=black](4)circle(0.05);
    \draw[fill=black](5)circle(0.05);
    \draw[fill=black](6)circle(0.05);
    \draw[fill=black](7)circle(0.05);

    \draw (1) -- (3);
    \draw (6) -- (3);
    \draw (1) -- (6);
    \draw (1) -- (5);
    \draw (1) -- (3);
    \draw (7) -- (3);
    \draw (2) -- (6);
    \draw (4)circle(0.577);

    \node[xshift=-0.55cm, yshift=0.1cm] at (1) {$w(xy)$};
    \node[xshift=-0.3cm, yshift=0.1cm] at (2) {$wy$};
    \node[xshift=-0.3cm, yshift=0.1cm] at (3) {$x$};
    \node[xshift=-0.12cm, yshift=0.25cm] at (4) {$y$};
    \node[xshift=-0.0cm, yshift=-0.3cm] at (5) {$wx$};
    \node[xshift=0.3cm, yshift=0.1cm] at (6) {$w$.};
    \node[xshift=0.3cm, yshift=0.1cm] at (7) {$xy$};

    \draw[-Triangle] (3) -- (-0.5,0);
    \draw[-Triangle] (3) -- (-0.6,0.23);
    \draw[-Triangle] (1) -- (0,1.232);
    \draw[-Triangle] (1) -- (-0.25,1.299);
    \draw[-Triangle] (6) -- (0.6,0.23);
    \draw[-Triangle] (6) -- (0.75,0.433);
    \draw[-Triangle] (0.201,1.1196) -- (0.2, 1.12);
    
\end{tikzpicture}
\end{equation}
Given this basis, a different choice of vector $w$ must be given by $w' = aw + bwx + cwy + dw(xy)$, where $a,b,c,d\in\R$. Using Diagram (\ref{Bfano}), we find that
\begin{align*}
    w'x &= -b w + a wx -d wy + c w(xy) \\
    w'y &= -cw + d wx + a wy - b w(xy) \\
    w'(xy) &= -d w -c wx + b wy +a w(xy).
\end{align*}
Let $\psi = \psi_{[w,w(xy)],t}\cdot \psi_{[wx,wy],t}$ and $\psi'=\psi_{[w',w'(xy)],t}\cdot \psi_{[w'x,w'y],t}$. It is then straight forward to check that $\psi$ and $\psi'$ act the same way on all elements of $\mathcal{B}$. For example, we have
\begin{align*}
    \psi(w') &= a\psi(w) + b\psi(wx) + c\psi(wy) + d\psi(w(xy)) \\
    & = a(\cos{t}w + \sin{t}w(xy)) + b(\cos{t}wx + \sin{t}wy) \\
    & + c(-\sin{t}wx + \cos{t}wy) + d(-\sin{t}w + \cos{t}w(xy)) \\
    & = \cos{t}\left[a w + b wx + c wy + d w(xy) \right] 
     + \sin{t}\left[-dw -c wx + b wy +a w(xy) \right] \\
    & = \cos{t}w' + \sin{t}w'(xy) = \psi'(w').
\end{align*}
Thus, the map $f_7$ is well-defined.
\end{proof}

We now show how the map $f_7$ interacts with the groups $\Spin(7)$ and $SO(7)$. Let $q$ denote the double covering $\Spin(7) \longrightarrow SO(7)$. 

\begin{prop}\label{prop:spin_is_image}
The image of the map $f_7$ is $\Spin(7) \subset SO(8)$. 
Moreover, for all $([x,y],e^{it})\in \Gr{2}{7}\times S^1$, we have $(q \circ f_7)([x,y],e^{it}) = \psi_{[x,y],2t} \in SO(7)$.
\end{prop}
\begin{proof}
Let $g = \psi_{[x,y],2t}$ and $\psi_t = \psi_{[e_0,xy],[x,y],[wx,wy],[w,w(xy)]}$. It suffices to show that $g(\alpha)\psi_t(\beta) = \psi_t(\alpha\beta)$ for all $\alpha,\beta\in \mathcal{B}$, where $\mathcal{B}$ is the basis constructed in the proof of Lemma \ref{lem:f7_well_defined}. We first claim that $\alpha\psi_\frac{\pi}{2}(\beta) = \psi_\frac{\pi}{2}(\alpha\beta)$ for all $\alpha,\beta\in \mathcal{B}$ with $\alpha\not\in\{x,y\}$.
Observe that if $\beta \not\in \{ x,y \}$, then $\psi_\frac{\pi}{2}(\beta) = \beta(xy)$, and if $\beta\in \{x,y\}$, then $\psi_\frac{\pi}{2}(\beta) = -\beta(xy)$. Thus, we have
\begin{equation*}
    \alpha\psi_\frac{\pi}{2}(\beta) = \begin{cases}
        \;\;\;\alpha(\beta(xy)),\quad \beta\not\in\{ x,y \} \\
        -\alpha(\beta(xy)),\quad \beta\in\{ x,y \}.
    \end{cases}
\end{equation*}
Furthermore, Lemmas \ref{properties} and \ref{associativity} imply that
\begin{equation*}
    \alpha(\beta(xy)) = \begin{cases}
        \;\;\;(\alpha\beta)(xy),\quad \alpha=\beta,\; \alpha=xy,\; \beta=xy, \text{ or } \alpha\beta = \pm xy \\
        -(\alpha\beta)(xy),\quad \text{otherwise}.
    \end{cases}
\end{equation*}
Combined with the assumption that $\alpha \not\in \{ x,y \}$, we get that
\begin{equation*}
    \alpha\psi_\frac{\pi}{2}(\beta) = \left\{ \begin{matrix}
         \;\;\;(\alpha\beta)(xy),\quad \alpha\beta \not\in \{ \pm x, \pm y \} \\
        -(\alpha\beta)(xy),\quad \alpha\beta \in \{ \pm x, \pm y \}
    \end{matrix} \right\} = \psi_\frac{\pi}{2}(\alpha\beta).
\end{equation*}
This proves the claim. Then, for $\alpha \in \mathcal{B}\setminus\! \{x,y\}$, we have
\begin{align*}
    g(\alpha)\psi_t(\beta) &= \alpha(\cos{t}\beta + \sin{t}\psi_\frac{\pi}{2}(\beta)) = \cos{t}\alpha\beta + \sin{t}\alpha\psi_\frac{\pi}{2}(\beta) \\
    &= \cos{t}\alpha\beta + \sin{t}\psi_\frac{\pi}{2}(\alpha\beta) = \psi_t(\alpha\beta).
\end{align*}
It remains to show that $g(\alpha)\psi_t(\beta) = \psi_t(\alpha\beta)$ for $\alpha\in \{ x,y \}$. These 16 cases are easily checked by direct computation, for example
\begin{equation*}
    g(x)\psi_t(y) = (\cos{2t}x + \sin{2t}y)(-\sin{t}x + \cos{t}y) = 
    -\sin{t}e_0 + \cos{t}xy = \psi_t(xy).
\end{equation*}
This completes the proof.
\end{proof}

\begin{defn}
Let $f_5$ denote the restriction of $f_7$ to $\Gr{2}{5}\times S^1 \subseteq \Gr{2}{7}\times S^1$. We then define the map
    \begin{align*}
        f_7 \times f_5\colon \Gr{2}{7}\times S^1 \times \Gr{2}{5}\times S^1 &\longrightarrow \Spin(7) \\
        \left([x,y],e^{it},[x'y'],e^{it'}\right) &\longmapsto f_7\left([x,y],e^{it}\right) \cdot f_5\left([x',y'],e^{it'}\right).
    \end{align*}
\end{defn}

\begin{lemma}\label{lem:is_cobordism_class}
The map $f_7\times f_5$ admits a complex orientation.
In particular, $f_7\times f_5$ is a proper complex-oriented smooth map and represents an element in $MU^3(\Spin(7))$.
\end{lemma}
\begin{proof}
The fact that $g$ admits a complex orientation follows from the facts that $S^1$ is stably almost complex, $\Gr{2}{7}$ and $\Gr{2}{5}$ are almost complex, and $\Spin(7)$ is a compact Lie group.
\end{proof}


\section{Computations in homology}\label{sec:homology}

We recall from \cite[Corollary III.3.15 and Theorem IV.2.19]{mt} that, for $n=7,8$, there is an isomorphism $H^3(\Spin(n);\Z) \cong \Z$.
Let $\gamma_3 \in H^3(\Spin(n);\Z)$ be a generator.
By \cite[Proposition 3.4]{KQ}, $\gamma_3$ is not contained in the image of the Thom morphism. 
However, we will now show that the Thom morphism sends the cobordism class represented by $f_7 \times f_5$ to an integer multiple of $\gamma_3$ in $H^3(\Spin(7);\Z)$. 

\begin{theorem}\label{thm:spin}
    The Thom morphism $MU^3(\Spin(7)) \rightarrow H^3(\Spin(7);\Z)$ maps the element represented by $f_7\times f_5$ to $\pm 8$ times the generator $\gamma_3\in H^3(\Spin(7);\Z)$.
\end{theorem}
\begin{proof}
\begin{sloppypar}
By Poincar\'e duality it suffices to show that the fundamental class of $\Gr{2}{7}\times S^1 \times \Gr{2}{5}\times S^1$ in homology is mapped to eight times the corresponding generator in $H_{18}(\Spin(7);\Z)$.
In \cite{KQ} we constructed a class of maps into special orthogonal groups. 
In particular, there is the map\\[-22pt]
\end{sloppypar}
\begin{align*}
    H_{7,0}\colon \Gr{2}{7}\times S^1 \times \Gr{2}{5}\times S^1 &\longrightarrow SO(7) \\
    \left([x,y],e^{it},[x',y'],e^{it'}\right) &\longmapsto \psi_{[x,y],t} \cdot \psi_{[x',y'],t'.}
\end{align*}
By \cite[Theorem 4.7]{KQ}, the induced map 
\begin{align*}
(h_{7,0})_\ast\colon H_{18}(\Gr{2}{7}\times S^1 \times \Gr{2}{5}\times S^1) \longrightarrow H_{18}(SO(7))
\end{align*}
is given by a multiplication by $\pm 4$ after identifying each homology group with $\Z$. 
We define the map
\begin{align*}
    p\colon \Gr{2}{7} \times S^1 \times \Gr{2}{5} \times S^1 &\longrightarrow\Gr{2}{7} \times S^1 \times \Gr{2}{5} \times S^1 \\
    \left( [x,y],e^{it},[x',y'],e^{it'} \right) &\longmapsto \left( [x,y],e^{2it},[x',y'],e^{2it'} \right).
\end{align*}
By Proposition \ref{prop:spin_is_image}, we get that $c \circ (f_7\times f_5) = h_{7,0}\circ p$. Thus, we have the commutative diagram
\begin{equation*}
\begin{tikzcd}[column sep=huge]
    H_{18}(\Gr{2}{7}\times S^1 \times \Gr{2}{5}\times S^1) \arrow[r, "(f_7\times f_5)_\ast"] \arrow[d, "p_\ast"'] & H_{18}(\Spin(7)) \arrow[d, "c_\ast"] \\
    H_{18}(\Gr{2}{7}\times S^1 \times \Gr{2}{5}\times S^1) \arrow[r, "(h_{7,0})_\ast"] & H_{18}(SO(7)),
\end{tikzcd}    
\end{equation*}
where all the homology groups are isomorphic to $\Z$. 
Since $p$ consists of two maps of degree $2$ on the circles, we see that $p_\ast$ is a multiplication by $4$. 
Furthermore, it follows from \cite[7.4]{pittie} that $c_\ast$ is a multiplication by $2$. Thus, we conclude that $(f_7\times f_5)_\ast$ is given by a multiplication by $\pm 8$, which completes the proof.
\end{proof}

Since $\Spin(8)$ is homeomorphic to $\Spin(7)\times S^7$ (see \cite{imn}), Theorem \ref{thm:spin} implies the following result.

\begin{cor}\label{cor:spin8}
The Thom morphism $MU^3(\Spin(8))\rightarrow H^3(\Spin(8);\Z)$ maps the element represented by
\begin{align*}
        \Gr{2}{7}\times S^1 \times \Gr{2}{5}\times S^1 \times S^7 &\longrightarrow \Spin(7)\times S^7 \cong \Spin(8) \\
        \left([x,y],e^{it},[x',y'],e^{it'}\!,s\right) &\longmapsto \left(f_7\left([x,y],e^{it}\right) \cdot f_5\left([x',y'],e^{it'}\right),s\right)
\end{align*}
to $\pm 8$ times the generator $\gamma_3\in H^3(\Spin(8);\Z)$.
\end{cor}

\begin{remark}
    By studying the differentials in the Atiyah--Hirzebruch spectral sequence for $\Spin(7)$, one can see that $2\cdot\gamma_3\in H^3(\Spin(7);\Z)$ is in the image of the Thom morphism. Thus, the construction of $f_7 \times f_5$ is not the best possible one. Furthermore, using the same method as in the proof of Theorem \ref{thm:spin}, we see that the element of $MU^{10}(\Spin(7))$ represented by $f_7$ is mapped to $\pm 4$ times the generator of $H^{10}(\Spin(7);\Z)$. However, there must exist some element of $MU^{10}(\Spin(7))$ which maps to $2$ times this generator.
\end{remark}


%
\bibliographystyle{amsalpha}

\end{document}